\documentclass{amsart}
\usepackage{graphicx}
\usepackage{subfig}
\usepackage{subcaption}
\usepackage{amsmath}
\usepackage{amssymb}
\usepackage{array}
\usepackage{amsthm}
\usepackage{bbm}
\usepackage{chngcntr}
\usepackage{algorithm}
\usepackage{algpseudocode}
\usepackage{caption}
\usepackage{enumitem}
\usepackage{bm}
\usepackage{amsfonts}
\usepackage{mathrsfs}
\usepackage{indentfirst}
\usepackage[mathscr]{eucal}
\usepackage{titlesec}
\usepackage{xcolor}
\renewenvironment{abstract}{
    \noindent\textbf{Abstract}\par
    \vspace{1em}\noindent\ignorespaces}
    {\par\noindent\ignorespacesafterend}

  \titleformat{\section}
  {\normalfont\Large\bfseries}
  {\thesection.}
  {0.5em}
  {}
\titleformat{\section}
  {\normalfont\normalsize\bfseries} 
  {\thesection.}
  {1em}
  {}
\titleformat{\subsection}
  {\normalfont\normalsize\bfseries} 
  {\thesubsection.} 
  {1em}
  {}
\newtheorem{theorem}{Theorem}[section]
\newtheorem{lemma}[theorem]{Lemma}
\newtheorem{proposition}[theorem]{Proposition}
\newtheorem{corollary}[theorem]{Corollary}
\numberwithin{equation}{section}
\counterwithin{figure}{subsection}
\counterwithin{table}{subsection}
\counterwithin{algorithm}{section}
\counterwithin{theorem}{section}
\makeatletter

\newcommand{\Rmnum}[1]{\expandafter\@slowromancap\romannumeral #1@}

\title[POD-Greedy Algorithms]{Some New Convergence Analysis and Applications of POD-Greedy Algorithms}
\author{Yuwen Li and Yupeng Wang}
\address{School of Mathematical Sciences, Zhejiang University, Hangzhou, Zhejiang 310058, China}
\email{liyuwen@zju.edu.cn}
\email{yupengw@zju.edu.cn}

\begin{document}
\maketitle
\begin{abstract}
In this article, we derive a novel convergence estimate for the weak POD-Greedy method with multiple POD modes and variable greedy thresholds in terms of the entropy numbers of the parametric solution manifold. Combining the POD with the Empirical Interpolation Method (EIM), we also propose an EIM-POD-Greedy method with entropy-based convergence analysis for simultaneously approximating parametrized target functions by separable approximants. Several numerical experiments are presented to demonstrate the effectiveness of the proposed algorithm compared to traditional methods. 
\end{abstract} 

\vspace{0.5em}\noindent\textbf{Keywords:} reduced basis method, POD-Greedy method, empirical interpolation method, entropy number, Kolmogorov width \par\addvspace\baselineskip

\vspace{0.5em}\noindent\textbf{MSC codes:} 
41A46, 41A65, 65J05, 65M12, 65D15\par\addvspace\baselineskip

\section{Introduction}
The Reduced Basis Method (RBM) is a popular numerical method for solving parametric Partial Differential Equations (PDEs) by accurately capturing essential features of solutions of the high-fidelity model through an easy-to-solve reduced model (see \cite{CohenDeVore2015,Hesthaven2022,Maday2006,Quarteroni2016,Veroy2012posteriori}). In general, RBMs consist of a costly offline stage and an efficient online stage. First the offline module constructs a low-dimensional reduced basis subspace approximating the high-fidelity solution manifold. Then
in the online stage, queried reduced basis solutions are rapidly computed in the reduced basis subspace, allowing for real-time numerical simulations. Common offline procedures for constructing reduced basis subspaces include the reduced basis greedy algorithm (cf.~\cite{BinevCohenDahmenDeVore2011,RozzaHuynhPatera2008}) and the Proper Orthogonal Decomposition (POD)(cf.~\cite{AtwellKing2004,KunischVolkwein2001}). Interested readers are also referred to other reduced order modeling techniques developed in, e.g., \cite{AzaiezBelgacem2018,CohenDahmen2020,LeeCarlberg2020,LiZikatanovZuo2024SISC,LiZikatanovZuo2024arXiv} etc.

For time-dependent problems, a standard RBM needs to uniformly approximate high-fidelity solution trajectories for each parameter. However, including all the snapshots along these trajectories for subspace construction is computationally infeasible due to the rapid growth in the dimension of the reduced basis subspace. A more practical approach is to gradually enlarge the reduced basis subspace by adding the first $m$ POD modes of the projection error at each greedy iteration, known as the POD-Greedy method (see \cite{Haasdonk2013,HaasdonkOhlberger2008}). In a different approach, the works \cite{DahmenPleskenWelper2014,UrbanPatera2014,Yano2014} developed  space-time RBMs by formulating dynamical PDEs as one-dimension higher stationary models. 

Classical convergence analysis of reduced basis greedy algorithms is based on the Kolmogorov $n$-width of the underlying solution manifold  (see \cite{BennerGugercinWillcox2015,BinevCohenDahmenDeVore2011,Buffa2012,Wojtaszczyk2015}). Following this line, \cite{Haasdonk2013} developed convergence rates of the weak POD-Greedy method for evolutionary problems in terms of the Kolmogorov width. Motivated by the recent work \cite{LiSiegel2024}, we shall present a novel error estimate of the weak POD-Greedy method using the entropy numbers of the convex hull of the solution manifold. Our analysis provides a direct comparison between the POD-Greedy error and the entropy numbers, while the classical result in \cite{Haasdonk2013} is a rate comparison. In addition, we are able to analyze the convergence of the weak POD-Greedy method with variable greedy threshold and multiple POD modes at each greedy step.

The efficiency of RBMs as well as other model reduction techniques e.g. \cite{BennerGugercinWillcox2015}, relies on the presence of affinely parametrized structures, which are not always available in practice. To address this issue, the Empirical Interpolation Method (EIM) \cite{BarraultMadayNguyenPatera2004} was developed for obtaining perturbed high-fidelity models with such affine parameter structures, see also, e.g., \cite{ChaturantabutSorensen2010,DrohmannHaasdonk2012,NegriManzoniAmsallem2015,NguyenPatera2008,Saibaba2020} for generalizations of the EIM. For a time-dependent parametrized family of coefficients, we shall develop an EIM-POD-Greedy method by combining the POD for temporal compression with the EIM in space direction, achieving higher efficiency for time-dependent simultaneous approximation of target functions. In addition, motivated by the framework in \cite{Li2024CGA}, we derive convergence rates of the proposed EIM-POD-Greedy method using the entropy numbers of the set of target functions, see Sections \ref{sec:PODEIM} and \ref{sec:NumExp} for details.

The rest of the paper is organized as follows. In Section \ref{sec:WeakPODGA}, we introduce the formulation of the weak POD-Greedy method. In Section \ref{sec:ConvergenceEntropy}, we derive the convergence rate analysis of the weak POD-Greedy method based on entropy numbers. Section \ref{sec:PODEIM} is devoted to the EIM-POD-Greedy method and its convergence analysis. In Section \ref{sec:NumExp}, we provide numerical experiments to illustrate the performance of the proposed algorithms.

\section{Weak POD-Greedy Method}\label{sec:WeakPODGA}
Let $V$ be a real Hilbert space equipped with the inner product $\langle\bullet, \bullet\rangle$ and norm $\|\bullet\|$. Given $T>0$ and an integer $J>0$, let $0=t_0<t_1<\cdots<t_J=T$ be a grid of the time interval $[0,T]$ with $t_j=j\tau$ and the step-size $\tau:=T/J$. For the set $\mathbb{I}:=\left\{t_j\right\}_{j=0}^J$ of time grid points,  
by $V_T:=L^2(\mathbb{I};V)$ we denote the Cartesian product space $V^{J+1}$ equipped with the inner product 
$$\langle u,v\rangle_{V_T}:=\sum_{j=0}^J\tau\langle u^j,v^j\rangle,$$ 
where $u=(u^0,u^1,\ldots,u^J)$, $v=(v^0,v^1,\ldots,v^J)$. Clearly $L^2(\mathbb{I};V)$ is a discretization of the Bochner-type space $L^2(0,T;V)$.

Given a $n$-dimensional subspace $V_n\subset V$,
let $P_{V_n}: V\rightarrow V_n$ be the orthogonal projection onto $V_n$. Similarly, $P_{V_{T,n}}: V_T\rightarrow V_{T,n}:=L^2(\mathbb{I};V_n)$ is the orthogonal projection onto $V_{T,n}$ with respect to $\langle\bullet,\bullet\rangle_{V_T}$. It is straightforward to verify that for $v\in V_T$, 
\begin{align*}
(P_{V_{T,n}}v)^j=P_{V_{n}}v^j,\quad j=0, 1, \ldots, J.
\end{align*}
To present POD-Greedy-type algorithms and convergence analysis, it is necessary to introduce a compact space-time solution manifold $\mathcal{M}_{T}\subset V_T$ of a parametric evolutionary PDE. We also consider the set $\mathcal{M}\subset V$ built upon slices of $\mathcal{M}_T$:
\begin{equation*}
\mathcal{M}=\left\{v^j: v\in\mathcal{M}_T,~j=0, 1, \ldots, J\right\}.
 \end{equation*}

\subsection{Example}\label{subsec:example}
On a physical domain $\Omega\subset\mathbb{R}^p$, a model problem of the POD-Greedy method is the following parametric parabolic problem: 
\begin{subequations}\label{parabolic}
\begin{align}
\partial_t u_{\mu}-\nabla\cdot(a_{\mu}\nabla u_{\mu})&=f\quad {\rm in}\quad\Omega\times(0,T],\\
u_{\mu}&=0\quad {\rm on} \quad\partial\Omega\times(0,T],\\
u_{\mu}&=g\quad {\rm on}\quad\Omega\times \left\{t=0\right\},
\end{align}
\end{subequations}
where the coefficient $a_{\mu}: [0,T]\rightarrow L^\infty(\Omega)$ depends on a varying parameter $\mu\in\mathcal{P}\subset\mathbb{R}^d$. For simplicity, we assume that \eqref{parabolic} is semi-discretized by the implicit Euler method in time with the numerical solution $u_{\tau,\mu}\approx u_\mu$.
In this case,  
\begin{equation*}
V=H_0^1(\Omega),\quad\mathcal{M}_T=\{u_{\tau,\mu}\in L^2(\mathbb{I};V): \mu\in\mathcal{P}\},
\end{equation*} 
and $u_{\tau,\mu}\in L^2(\mathbb{I};V)$ satisfies $u_{\tau,\mu}^0=u_0$ and 
\begin{equation}\label{semidiscrete}
\frac{u_{\tau,\mu}^j-u_{\tau,\mu}^{j-1}}{\tau}-\nabla\cdot(a_\mu(t_j)\nabla u^j_{\tau,\mu})=f(t_j), \quad1\leq j\leq J.
\end{equation}
If the space is further discretized by a finite element space $V_h\subset H_0^1(\Omega)$, then
\begin{equation*}
    V=V_h,\quad\mathcal{M}_T=\{u_{\tau,h,\mu}\in L^2(\mathbb{I};V_h): \mu\in\mathcal{P}\}.
\end{equation*}
Let $(\bullet,\bullet)$ denote the $L^2(\Omega)$ inner product. The fully discrete solution $u_{\tau,h,\mu}\in L^2(\mathbb{I};V_h)$ of \eqref{parabolic} solves
\begin{subequations}\label{fullydiscrete}
\begin{align}
\Big(\frac{u_{\tau,h,\mu}^j-u_{\tau,h,\mu}^{j-1}}{\tau},v_h\Big)+(a_\mu(t_j)\nabla u_{\tau,h,\mu}^j,\nabla v_h) &= (f(t_j),v_h),\\
(u_{\tau,h,\mu}^0-g,v_h)&=0,
\end{align}
\end{subequations}
for all $v_h\in V_h$ and $1\leq j\leq J$.

\subsection{POD-Greedy Method}
For a sequence $v=(v^0,v^1,\ldots,v^J)\in V_T$, the POD method aims to find a low-dimensional subspace that best approximates the snapshots $v^0, v^1, \ldots, v^J$ in the total energy norm $\|\cdot\|_{V_T}$. To introduce the POD, we need the correlation operator $C_v: V\rightarrow V$ defined by
\begin{equation}\label{operatorC}
C_v(w)=\sum_{j=0}^J \tau \langle v^j,w \rangle v^j.
\end{equation}
Clearly $C_v$ is a symmetric and positive-definite operator with range $R(C_v)={\rm span}\{v^0, v^1, \ldots, v^J\}$. Let $\lambda_0(v)\geq\lambda_1(v)\geq\cdots\geq\lambda_{J}(v)\geq 0$
be the largest $J$ eigenvalues of $C_v$ and $\varphi_i(v)$ a unit eigenvector associated to $\lambda_i(v)$.
For an integer $N\ll J$, the $N$-dimensional POD subspace is set as $W_N={\rm span}\left\{\varphi_0(v), \ldots, \varphi_{N-1}(v)\right\}$. The POD subspace satisfies the best low-rank approximation property
\begin{equation*}
    \sum_{j=0}^J\big\|v^j-P_{W_N}v^j\big\|^2=\min_{{\rm dim}W=N} \sum_{j=0}^J\big\|v^j-P_Wv^j\big\|^2=\sum_{i\geq N}\lambda_i(v).
\end{equation*}

%Now we are in a position to present the weak POD-Greedy method for efficiently simultaneously solving parametrized time-dependent PDEs such as \eqref{parabolic} for arabitrary parameter and time-discretization. 
Now we are in a position to present the weak POD-Greedy method for efficiently solving parametrized time-dependent PDEs, such as \eqref{parabolic} corresponding to many instances of the  parameter $\mu$.

\begin{algorithm}[thp]
\caption{Weak POD-Greedy Method}\label{alg:PODGA}       
\begin{algorithmic}[1]
\Statex \textbf{Input:} two integers $N$, $m$, a set of threshold constants $\{\gamma_n\}_{n\geq1}\subset (0,1]$;
 \Statex \textbf{Initialization:} $V_0=\{0\}, V_{T,0}=L^2 (\mathbb{I};V_0)$;  
 \Statex \textbf{For} {$n=1:N$}

       \Statex $\qquad$ select $u_n\in \mathcal{M}_T$ such that 
       \begin{equation*}
      \|u_n-P_{V_{T,n-1}}u_n\|_{V_T}\geq\gamma_n\sup\limits_{u\in \mathcal{M}_T}\|u-P_{V_{T,n-1}}u\|_{V_T};  
    \end{equation*}

     \Statex $\qquad$ set $r_n:=u_n-P_{V_{T,n-1}}u_n$; compute the $m$ leading eigen-pairs 
     \Statex $\qquad$ $\left\{(\lambda_n^1, f_n^1),\ldots,(\lambda_n^m, f_n^m)\right\}$ of $C_{r_n}$ with $\lambda_n^1\geq\cdots\geq\lambda_n^m$ and
     \[
     C_{r_n}(f_n^k)=\lambda_n^k f_n^k,\quad 1\leq k\leq m;
     \]

     \Statex $\qquad$ update the reduced basis subspace
    \begin{align*}
            V_n&:=V_{n-1}\oplus {\rm span} \left\{f_n^1,\ldots,f_n^m \right\},\\
            V_{T,n}&:=L^2 (\mathbb{I},V_n);
    \end{align*}
   \Statex \textbf{EndFor}
   \Statex \textbf{Output:} the reduced basis subspace $V_{T,N}$. 
   \end{algorithmic} 
\end{algorithm}

Compared with \cite{Haasdonk2013},  Algorithm \ref{alg:PODGA} makes use of a variable threshold $\gamma_n$ and enriches the reduced basis subspace by multiple POD modes at each iteration. The error of Algorithm \ref{alg:PODGA} is
\begin{equation*}
\sigma_n:=\sup_{u\in\mathcal{M}_T}\|u-P_{V_{T,n-1}}u\|_{V_T},
\end{equation*}
which is non-increasing as $n$ grows. In general, the weak greedy criterion 
\begin{equation}\label{greedy}
   \|u_n-P_{V_{T,n-1}}u_n\|_{V_T}\geq\gamma_n\sup_{u\in \mathcal{M}_T}\|u-P_{V_{T,n-1}}u\|_{V_T}
\end{equation}
in Algorithm \ref{alg:PODGA} is replaced by the practical maximization problem
\begin{equation}\label{estimator}
   u_n=\arg\max_{u\in \mathcal{M}_T}\Delta_{n-1}(u)
\end{equation}
where $\Delta_{n-1}(u)$ is an economic a posteriori error estimator that fulfills \begin{equation}\label{upperlowerbound}
    c_{n,1}\Delta_{n-1}(u)\leq\|u-P_{V_{T,n-1}}u\|_{V_T}\leq c_{n,2}\Delta_{n-1}(u)
\end{equation}
with $c_{n,1}, c_{n,2}>0$ being modest constants, 
see Section \ref{subsec:PODGreedyExp} for details. It is to see that \eqref{estimator} implies \eqref{greedy} with $\gamma_n=c_{n,1}/c_{n,2}$. 

Once $V_{T,N}$ is constructed, the parametric family of PDEs is approximately solved by the online module of the POD-Greedy method within $V_{T,N}$ to save computational cost. For example,
the numerical solution $u_{\tau,N,\mu}\in V_{T,N}$ of the POD-Greedy method for \eqref{parabolic} is determined by 
\begin{subequations}\label{reducedmodel}
\begin{align}
\Big(\frac{u_{\tau,N,\mu}^j-u_{\tau,N,\mu}^{j-1}}{\tau},v_N\Big)+(a_\mu(t_j)\nabla u_{\tau,N,\mu}^j,\nabla v_N) &= (f(t_j),v_N),\\
(u_{\tau,N,\mu}^0-g,v_N)&=0,
\end{align}
\end{subequations}
for any $v_N\in V_N$ and $1\leq j\leq J$.

\subsection{Properties of the POD-Greedy Method} 
In the end of this section, we summarize some properties of the POD-Greedy method that will be used in the error analysis. 
\begin{lemma}\label{sumlambda}
  For any $v=(v^0,\ldots,v^J)\in L^2(\mathbb{I};V)$, the eigenvalues of $C_v$ satisfy 
    \begin{equation*}
        \sum\limits_{i=0}^J\lambda_i (v)=\|v\|^2_{V_T}.
    \end{equation*}
\end{lemma}
\begin{proof}
The definition of $C_v$ with $\varphi_i=\varphi_i(v)$ implies
\begin{equation*}
    \lambda_i(v)=\langle C_v(\varphi_i),\varphi_i\rangle=\sum\limits_{j=0}^J\tau|\langle v^j,\varphi_i\rangle|^2.
\end{equation*}
The eigenvectors  $\{\varphi_i\}_{0\leq i\leq J}$ of the symmetric operator $C_v$ form an orthonormal basis of its range. As a result, we have
    \begin{equation*}
       \sum\limits_{i=0}^J\lambda_i(v) =\sum\limits_{j=0}^J\tau\sum_{i=0}^J|\langle v^j,\varphi_i\rangle|^2=\tau\sum_{j=0}^J\|v^j\|^2,
    \end{equation*}
which completes the proof.
\end{proof}

Next we check a property of the operator $C_{r_n}$ (see \eqref{operatorC}), where $r_n=u_n-P_{V_{T,n-1}}u_n$ is the residual in the weak POD-Greedy method in Algorithm \ref{alg:PODGA}. For $1\leq k\leq m$, direct calculation shows
\begin{equation}\label{eigenpairrelation}
\begin{aligned}
    \tau\sum_{j=0}^J\langle r_n^j,f_n^k \rangle u_n^j&=\tau\sum_{j=0}^J \langle r_n^j ,f_n^k \rangle P_{V_{n-1}}u_n^j+\tau\sum_{j=0}^J \langle r_n^j ,f_n^k \rangle r_n^j\\
    &=\tau\sum\limits_{j=0}^J \langle r_n^j ,f_n^k \rangle P_{V_{n-1}}u_n^j+\lambda_n^k f_n^k.
\end{aligned}
\end{equation}
By the orthogonality of $\{f_i^k\}_{1\leq i\leq n-1,1\leq k\leq m}$, we obtain 
\begin{equation*}
    P_{V_{n-1}}u_n^j=\sum_{i=1}^{n-1}\sum_{k=1}^{m} \langle u_n^j,f_i^k \rangle  f_i^k
\end{equation*}
and thus the following identity
\begin{equation}\label{vik}
\begin{aligned}
v_n^k&:=\tau\sum_{j=0}^J\langle r_n^j ,f_n^k\rangle u_n^j\\
&=\tau\sum_{i=1}^{n-1}\sum_{s=1}^{m}\sum_{j=0}^J\langle r_n^j,f_n^k\rangle  \langle u_n^j,f_i^s\rangle f_i^s+\lambda_n^k f_n^k.
\end{aligned}
\end{equation}
Using the matrix-vector notation, the above identity translates into 
\begin{equation}\label{vAf}
    \begin{pmatrix}
     v_1^1 \\
     v_1^2 \\
     \vdots\\
     v_n^{m} 
     \end{pmatrix}=
     A\begin{pmatrix}
f_1^1  \\
f_1^2  \\
\vdots \\
f_n^{m}\\
\end{pmatrix}, 
\end{equation}
where $A=(a_{ij})\in \mathbb{R}^{mn\times mn}$ is a lower triangular matrix given by
\begin{equation}\label{matrixA}
    A=\begin{pmatrix}
        &{A_1} &{0} &{\cdots} &{0}\\
        &{B_2^1} &{A_2} &{\cdots} &{0}\\
        &{\vdots} &{\vdots} &{\ddots} &{\vdots}\\
        &{B_n^1} &{B_n^2} &{\cdots} &{A_n}
    \end{pmatrix},
\end{equation}
where $A_i={\rm diag}(\lambda_i^1, \lambda_i^2, \ldots, \lambda_i^m)$ for $1\leq i\leq n$
and 
\begin{equation*}
    B_i^s=\tau\sum_{j=0}^J\begin{pmatrix}
    &{\langle r_i^j,f_i^1\rangle\langle u_i^j, f_s^1\rangle} &{\langle r_i^j,f_i^1\rangle\langle u_i^j,f_s^2\rangle} &{\cdots} &{\langle r_i^j,f_i^1\rangle\langle u_i^j,f_s^m\rangle}\\
    &{\langle r_i^j,f_i^2\rangle\langle u_i^j,f_s^1\rangle} &{\langle r_i^j,f_i^2\rangle\langle u_i^j,f_s^2\rangle} &{\cdots} &{\langle r_i^j,f_i^2\rangle\langle u_i^j,f_s^m\rangle}\\
    &{\vdots} &{\vdots} &{\ddots} &{\vdots}\\
    &{\langle r_i^j,f_i^m\rangle\langle u_i^j,f_s^1\rangle} &{\langle r_i^j,f_i^m\rangle\langle u_i^j,f_s^2\rangle} &{\cdots} &{\langle r_i^j,f_i^m\rangle\langle u_i^j,f_s^m\rangle}
    \end{pmatrix}
\end{equation*}
for each $2\leq i\leq n$ and $1\leq s\leq n-1$. In addition, we have the following identity by \eqref{vAf},
\begin{align}\label{subspaceequal}
   V_i:={\rm span}\left\{v_1^1,v_1^2,\ldots,v_i^m\right\}={\rm span}\left\{f_1^1,f_1^2,\ldots,f_i^m\right\},\quad1\leq i\leq n.
\end{align}

\section{Entropy-Based Convergence Estimate}\label{sec:ConvergenceEntropy}
In this section, we derive a novel convergence estimate for the weak POD-Greedy method in Algorithm \ref{alg:PODGA}. Throughout the rest of this paper, by $A\lesssim B$ (resp.~$A\eqsim B$) we mean $A\leq CB$ (resp.~$A\lesssim$ and $B\lesssim A$), where $C$ is a positive generic constant that may change from line to line but independent of the parametric function class $\mathcal{M}_T$.

\subsection{Entropy Numbers and Convergence Analysis}
Given a set $K$,
the symmetric convex hull of $K$ is 
\begin{align*}
    {\rm co}(K):=\overline{\left\{\sum_i c_i g_i:c_i\in\mathbb{R},\sum_i |c_i|\le 1, g_i\in K \ \text{for each}\ i\right\}}.
\end{align*}
Our analysis hinges on the entropy numbers of ${\rm co}(K)$: 
\begin{align*}
    \varepsilon_n({\rm co}(K)):=\inf\left\{\varepsilon>0: {\rm co}(K)\ \text{is}\ \text{covered}\ \text{by}\ 2^n\ \text{balls}\ \text{of}\ \text{radius}\ \varepsilon\right\}.
\end{align*}
When $K$ is compact, $\varepsilon_n({\rm co}(K))$ converges to zero as $n\to\infty$ (see, e.g., \cite{Lorentz1996}).
Let $S_n$ denote the volume of a unit ball in $\mathbb{R}^n$. Next we introduce a key entropy-based lemma which is essentially proved in \cite{LiSiegel2024}.
\begin{lemma}\label{lemma:comparison}
For the lower triangular system \eqref{vAf}, it holds that
\begin{align*}
   \left (\prod_{i=1}^n\prod_{k=1}^m\frac{\lambda_i^k}{\tau\sum_{j=0}^J|\langle r_i^j,f_i^k\rangle|}\right)^{\frac{1}{mn}}\leq \big ((mn)!S_{mn}\big)^{\frac{1}{mn}} \varepsilon_{mn}({\rm co}(\mathcal{M})).
\end{align*}
\end{lemma}
\begin{proof}
We consider the symmetric convex set \[
S={\rm co}\left(\left\{\frac{v_1^1}{\tau\sum_{j=0}^J |\langle r_1^j,f_1^1\rangle|}, \ldots, \frac{v_n^m}{\tau\sum_{j=0}^J |\langle r_n^j,f_n^m \rangle|}\right\}\right).\]
After checking the definition of $v_i^k$ in \eqref{vik}, we observe that
\begin{align*}
    \frac{v_i^k}{\tau\sum_{j=0}^J|\langle r_i^j,f_i^k\rangle|}\in {\rm co}(\mathcal{M}),\quad1\leq i\leq n,\quad1\leq k\leq m,
\end{align*}
which implies that $S\subset {\rm co}(\mathcal{M})$ and 
\begin{align*}
 \varepsilon_n(S)\leq \varepsilon_n({\rm co}(\mathcal{M})).
\end{align*} 
We note that $S$ can be identified as a skew simplex in the Euclidean space $\mathbb{R}^{mn}$ via a bijection from $V_n={\rm span}\left\{f_1^1,f_1^2,\ldots,f_n^m\right\}$ to $\mathbb{R}^{mn}$. In this way, one could compute its volume 
\begin{equation*}
    {\rm vol}(S)=\frac{2^{mn}}{(mn)!}\prod_{i=1}^n\prod_{k=1}^m\frac{\|v_i^k-P_{V_{i-1}}v_i^k\|}{\tau\sum_{j=0}^J|\langle r_i^j,f_i^k\rangle|}.
\end{equation*}
On the one hand, it follows from the definition of $\varepsilon_{mn}:=\varepsilon_{mn}({\rm co}(\mathcal{M}))$ that $S$ can be covered by $2^{mn}$ balls of radius $\varepsilon_{mn}$, thus
\begin{equation}\label{volcomparison}
\frac{2^{mn}}{(mn)!}\prod_{i=1}^n\prod_{k=1}^m\frac{\|v_i^k-P_{V_{i-1}}v_i^k\|}{\tau\sum_{j=0}^J|\langle r_i^j,f_i^k\rangle|} \leq 2^{mn} S_{mn}\varepsilon_{mn}^{mn}.
\end{equation}
%On the other hand, by (\ref{vAf}) and (\ref{subspaceequal})
On the other hand, by the definition of $v_i^k$ in \eqref{vik}, we have
\begin{align*}
    v_i^k-P_{V_{i-1}}v_i^k &=v_i^k-\tau\sum_{t=1}^{i-1}\sum_{s=1}^m\sum_{j=0}^J\langle r_n^j,f_n^k\rangle\langle u_n^j,f_i^s\rangle f_t^s=\lambda_i^k f_i^k,
\end{align*}
which yields
\begin{equation}\label{projectionerror}
    \|v_i^k-P_{V_{i-1}}v_i^k\|=
    \lambda_i^k.
\end{equation}
Combining \eqref{volcomparison} and \eqref{projectionerror} completes the proof.
\end{proof}

Using Lemma \ref{lemma:comparison}, we can directly compare the RBM errors $\left\{\sigma_n\right\}_{n\geq 1}$ with the entropy numbers of ${\rm co}(\mathcal{M})$.

\begin{theorem}\label{thm:sigmaepsilon}
Let $\theta_n:=\lambda_n^m/\lambda_n^1$ in Algorithm \ref{alg:PODGA}. For each  $n\geq1$, it holds that
    \begin{align*} \sigma_n\leq\Big(\prod_{i=1}^n\gamma_i\sqrt{\theta_i}\Big)^{-\frac{1}{n}}\sqrt{(J+1)T}((mn)!S_{mn})^{\frac{1}{mn}}\varepsilon_{mn} ({\rm co}(\mathcal{M})).
    \end{align*}
\end{theorem}
\begin{proof} 
First we note that for $1\leq i\leq n$ and $1\leq k\leq m$,
\begin{equation}\label{thm:sigmaepsilon:step1}
\begin{aligned}
\sum_{j=0}^J\tau|\langle r_i^j,f_i^k\rangle|&\leq\sqrt{\sum_{j=0}^J \tau}\cdot\sqrt{\sum_{j=0}^J\tau|\langle r_i^j,f_i^k\rangle|^2}\\
&=\sqrt{T|\langle C_{r_i}(f_i^k),f_i^k\rangle|}=\sqrt{T\lambda_i^k}.
\end{aligned}
\end{equation}
Using \eqref{greedy} and Lemma \ref{sumlambda}, we have
\begin{equation}\label{thm:sigmaepsilon:step2}
\lambda_i^k \geq\theta_i\lambda_i^1\geq \frac{\theta_i\gamma_i^2\sigma_i^2}{J+1}.
\end{equation}
It then follows from \eqref{thm:sigmaepsilon:step1} and \eqref{thm:sigmaepsilon:step2} that 
\begin{align}\label{diagonal}
    \frac{\lambda_i^k}{\sum_{j=0}^J \tau|\langle r_i^j,f_i^k\rangle|}
    \geq \sqrt{\frac{\lambda_i^k}{T}}\geq \frac{\sqrt{\theta_i}\gamma_i\sigma_i}{\sqrt{(J+1)T}}.
\end{align}
Combining the above inequality with Lemma \ref{lemma:comparison} yields
\begin{align*}
\Big(\prod\limits_{i=1}^n\sigma_i^m\Big)^{\frac{1}{mn}}\leq\Big(\prod_{i=1}^n\gamma_i\sqrt{\theta_i}\Big)^{-\frac{1}{n}}\sqrt{(J+1)T}\big((mn)!S_{mn}\big)^{\frac{1}{mn}}\varepsilon_{mn} ({\rm co}(\mathcal{M})).
\end{align*}
Finally noting $\sigma_1\geq \sigma_2 \geq \cdots\geq\sigma_n$ completes the proof.
\end{proof}

Using Theorem \ref{thm:sigmaepsilon} and the well-known asymptotic formulae
\begin{align*}
\lim_{n\to\infty}\frac{n!}{\sqrt{2\pi n}(\frac{n}{{\rm e}})^n}=1,\quad\lim_{n\to\infty}\frac{S_n}{\frac{1}{\sqrt{\pi n}}(\frac{2\pi{\rm e}}{n})^{\frac{n}{2}}}=1,
\end{align*}
we obtain a transparent entropy-based convergence rate of the POD-Greedy method.
\begin{corollary}\label{cor:sigmaepsilon}
    There exists a constant $C=C(J,T,m)>0$ such that for $n\geq 1$, 
    \[
      \sigma_n\leq C\Big(\prod_{i=1}^n\gamma_i\sqrt{\theta_i}\Big)^{-\frac{1}{n}}\sqrt{n}\varepsilon_{mn} ({\rm co}(\mathcal{M})).
    \]
\end{corollary}
We note that Corollary \ref{cor:sigmaepsilon} does not imply convergence of the POD-Greedy method when the underlying sliced solution manifold is extremely massive, i.e., $\varepsilon_n({\rm co}(\mathcal{M}))=O(n^{-\alpha})$ with $\alpha\leq\frac{1}{2}$. However, such a pessimistic bound of entropy numbers  rarely happens in reduced order modeling of elliptic and parabolic equations, see \cite{CohenDeVore2015} and the end of Section \ref{sec:ConvergenceNwidth}.

In the worst-case scenario, the POD-Greedy method with $m\geq 2$ will not provide higher numerical accuracy than the case $m=1$. When $\lambda_n^2=0$, the inclusion of multiple modes would introduce redundancy without enhancing the approximation accuarcy of the reduced basis subspace. We shall make a thorough numerical comparison among POD-Greedy methods with different values of $m$ in Section \ref{sec:NumExp}.

Another natural question is whether the order of convergence in Corollary \eqref{cor:sigmaepsilon} is sharp or not. In the following, we give an affirmative answer to this question.  
\begin{proposition}
    There exists a compact set $\mathcal{M}_T\subset L^2 (\mathbb{I};\ell^2 (\mathbb{N}))$ such that the weak POD-Greedy method with $m=1$ and $\gamma_n\equiv1$ satisfies
    \begin{equation*}
        \sigma_n\eqsim\sqrt{n}\varepsilon_n ({\rm co}(\mathcal{M}))_{\ell^2(\mathbb{N})}.
    \end{equation*}
\end{proposition}
\begin{proof}
Let $V=\ell^2(\mathbb{N})$ and $(x_i)_{i=1}^\infty\in V$ be a sequence given by
\[
x_i=2^{-k\alpha}\quad\text{ when }2^{k-1}\leq i\leq 2^k-1
,k\geq 1\] 
for some $\alpha>0$ with $x_i\eqsim i^{-\alpha}$ by an asymptotic argument. Given $\lambda >0$, we consider the compact set 
\[
\mathcal{M}_T=\left\{u_i\in V^{J+1}: u_i=({\rm e}^{-\lambda t_0} x_i e_i, \ldots, {\rm e}^{-\lambda t_J} x_i e_i),~ i\geq 1\right\},
\]
where $e_i$ is the $i$-th unit vector in $\ell^2(\mathbb{N})$. 
By the orthogonality of $\left\{e_i\right\}_{i\geq 1}$, we indeed obtain the POD mode $f_n=e_n$ (or $f_n=-e_n$) at the $n$-th iteration of Algorithm \ref{alg:PODGA}. 
The weak POD-Greedy method satisfies
\begin{equation}\label{sigmancxn}
    \sigma_n=\big\|u_n-P_{V_{T,n-1}}u_n\big\|_{V_T}=c|x_n|\eqsim n^{-\alpha},
\end{equation}
where $c=\sqrt{\sum_{j=0}^J {\rm e}^{-2\lambda t_j}}$. For $0\leq j\leq J$, let $\mathcal{M}_j=\big\{e^{-\lambda t_j} x_i e_i:i\geq 1\big\}$ and recall that $\mathcal{M}:=\bigcup_{j=0}^J\mathcal{M}_j$. It follows from Proposition 3.5 of \cite{LiSiegel2024} that
\begin{equation*}
\varepsilon_n({\rm co}(\mathcal{M}_j))_{\ell^2(\mathbb{N})}\leq \varepsilon_n\big({\rm co}(\{x_i e_i:i\geq 1\})\big)_{\ell^2 (\mathbb{N})}\lesssim n^{-\frac{1}{2}-\alpha}.
\end{equation*}
The sub-additivity of the entropy numbers (cf.~\cite{Lorentz1996}) further implies that
\begin{equation*}
\varepsilon_{(J+1)n}({\rm co}(\mathcal{M}))_{\ell^2 (\mathbb{N})}\leq \sum\limits_{j=0}^J \varepsilon_n ({\rm co}(\mathcal{M}_j))_{\ell^2 (\mathbb{N})}\lesssim n^{-\frac{1}{2}-\alpha}.
 \end{equation*}
and thus 
\begin{equation}\label{epsilonalpha}
\varepsilon_n ({\rm co}(\mathcal{M}))_{\ell^2 (\mathbb{N})}\lesssim n^{-\frac{1}{2}-\alpha}.
 \end{equation}
Combining \eqref{sigmancxn}, \eqref{epsilonalpha} with Corollary \ref{cor:sigmaepsilon} completes the proof.
\end{proof}

%For example, if the second basis vector $(m=2)$ of one trajectory contains much less information than the first basis vector of the next trajectory, or in the extreme case where the entire trajectory can be adequately compressed using just one mode, i.e. $\lambda^1>0,$ and $\lambda^2=\cdots=\lambda^m=0$, then the additional POD modes will simply become scalings of the first basis vector. 

\subsection{Width-Based Convergence}\label{sec:ConvergenceNwidth}
Classical convergence analysis of RBMs is based on the Kolmogorov $n$-width of the solution manifold $\mathcal{M}\subset V$:
\begin{align*}
d_n(\mathcal{M}):=\inf_{\dim W=n}\sup_{v\in\mathcal{M}}\inf_{w\in W}\|v-w\|,
\end{align*}
which measures the best possible approximation accuracy of $\mathcal{M}$ by $n$-dimensional subspaces.
The work \cite{Haasdonk2013} shows that for each $\alpha>0$ and $a>0$, Algorithm \ref{alg:PODGA} with $m=1$ and constant thresholds $\gamma_1=\gamma_2=\cdots$ satisfies
\begin{align*}
    d_n(\mathcal{M})\lesssim n^{-\alpha}&\Longrightarrow \sigma_n\lesssim n^{-\alpha},\\
    d_n(\mathcal{M})\lesssim {\rm e}^{-an^{\alpha}}&\Longrightarrow \sigma_n\lesssim {\rm e}^{-bn^{\beta}}
\end{align*}
for $\beta=\frac{\alpha}{\alpha+1}$ and some $b>0$. 

It is noted that the entropy-based error estimates in Theorem \ref{thm:sigmaepsilon} and Corollary \ref{cor:sigmaepsilon} applies to the POD-Greedy method with a variable threshold $\gamma_n$. This flexibility is particularly useful for realistic analysis of the POD-Greedy method for evolutionary PDEs, due to the fact that the lower bound in \eqref{upperlowerbound} is generally not uniform for time-dependent problems, i.e., $c_{n,1}$ is dependent on $n$. Therefore, the proposed entropy-based convergence estimate is of  practical importance for allowing non-uniform a posteriori error bounds.

In the end of this section, we briefly discuss on the decay of the entropy numbers $\varepsilon_n({\rm co}(\mathcal{M}))$
for model equations such as \eqref{semidiscrete}. Rewriting \eqref{semidiscrete} we obtain
\begin{equation}
u_{\tau,\mu}^j-\tau\nabla\cdot(a_\mu(t_j)\nabla u^j_{\tau,\mu})=u_{\tau,\mu}^{j-1}+\tau f(t_j)
\end{equation}
for $j=1, 2, \ldots, J$, where each $u_{\tau,\mu}^j\in H_0^1(\Omega)$ solves a parametrized elliptic problem. Let  $\mathcal{P}\subset\mathbb{R}^d$ be a compact set of parameters and assume that
\begin{equation}\label{dna}
    d_n(\{a_\mu(t_j): j=0, 1, \ldots, J,~\mu\in\mathcal{P}\})\leq C(s)n^{-s}
\end{equation} 
for all $s>0$. Let  $\mathcal{M}_j=\{u_{\tau,\mu}^j\in H_0^1(\Omega): \mu\in\mathcal{P}\}$. We remark that the assumption \eqref{dna} is true for common family of parametric coefficients such as the piecewise constant $a_\mu$ used in Section \ref{sec:NumExp}.
Using the theorem relating the Kolmogorov width of coefficient sets to solution manifolds of elliptic problems in \cite{CohenDeVore2015,Cohen2015IMA} and induction, we successively obtain that $d_n(\mathcal{M}_1)\lesssim n^{-s}, \ldots, d_n(\mathcal{M}_J)\lesssim n^{-s}$ and thus $$d_n(\mathcal{M})\lesssim\sum_{j=0}^Jd_n(\mathcal{M}_j)\lesssim n^{-s}$$ for all $s>0$. By Carl's inequality in approximation theory (see \cite{Carl1981,Lorentz1996,LiLi2024REIM}), the decay of $d_n(\mathcal{M})=d_n({\rm co}(\mathcal{M}))$ implies the convergence rate of $\varepsilon_n({\rm co}(\mathcal{M}))$ of the same order:
\begin{equation*}
\varepsilon_n({\rm co}(\mathcal{M}))\lesssim n^{-s},\quad\text{ for all }s>0.
\end{equation*}
\section{EIM-POD-Greedy Method}\label{sec:PODEIM}

\begin{algorithm}[thp]
\caption{EIM-POD-Greedy Method}     
\label{alg:PODEIM}       
\begin{algorithmic}[3] 
\Statex\textbf{Input:} a parameter set $\mathcal{P}\subset\mathbb{R}^d$, a set of parametrized functions $\mathcal{F}_T=\left\{a_\mu\text{ on }\mathbb{I}\times\Omega: \mu\in\mathcal{P}\right\}$, a candidate set of interpolation points $\Sigma\subset\Omega$, two positive integers $N$ and $m$;  

\State \textbf{Initialization:} set $V_0=\{0\}$, $V_{T,0}=L^2 (\mathbb{I};V_0)$, $\Pi_{T,0}=0$, $\Pi_{0,m}=0$;
    
    \Statex\textbf {For} {$n=1:N$}
    
       select $a_{\mu_n}\in\mathcal{F}_T$ such that 
       \begin{equation*}
           \|a_{\mu_n}-\Pi_{T,n-1}a_{\mu_n}\|_{V_T}= \sup\limits_{a_\mu\in \mathcal{F}_T}\|a_\mu-\Pi_{T,n-1}a_\mu\|_{V_T};
       \end{equation*}
       
    for $r_n:=a_{\mu_n}-\Pi_{T,n-1} a_{\mu_n}$, compute the first $m$ eigen-pairs $(\lambda_n^1, f_n^1),\ldots, $
    \Statex\qquad  $(\lambda_n^m,f_n^m)$ of $C_{r_n}$ with $\lambda_n^1\geq\cdots\geq\lambda_n^m$, $C_{r_n}(f_n^k)=\lambda_n^k f_n^k$; set $\Pi_{n,0}:=\Pi_{n-1,m}$; 

    \Statex\qquad\textbf{For}{ $k=1:m$ }
       \Statex\qquad\qquad for the residual $r_n^k:=f_n^k-\Pi_{n,k-1}f_n^k$, select $x_{(n-1)m+k}\in\Sigma$ such that 
        \begin{align*} |r_n^k(x_{(n-1)m+k})|=\sup_{x\in\Sigma}|r_n^k(x)|;
        \end{align*}
        \Statex\qquad\qquad compute $q_{(n-1)m+k}:=r_n^k/r_n^k(x_{(n-1)m+k})$ and $$B_{n,k}:=(q_j(x_i))_{1\leq i,j\leq(n-1)m+k};$$
        
        \Statex\qquad\qquad define the interpolation $$\qquad\qquad\Pi_{n,k}g :=(q_1,\ldots,q_{(n-1)m+k})B_{n,k}^{-1}(g(x_1),\ldots,g(x_{(n-1)m+k}))^\top;$$
    \Statex \qquad\textbf{EndFor}
     
     \Statex\qquad set $V_n:={\rm span}\{q_1,\ldots,q_{mn}\}$, $V_{T,n}:=L^2(\mathbb{I}; V_n)$, and $\Pi_{T,n}: \mathcal{F}_T\rightarrow V_{T,n}$ as
    \begin{align*}
        (\Pi_{T,n}a_\mu)^j:= \Pi_{n,m} (a_\mu(t_j)),\quad0\le j \le J;
    \end{align*}
   \Statex\textbf{EndFor}
   \Statex\textbf{Output:} the EIM-POD-Greedy method interpolation $\Pi_N:=\Pi_{N,m}$ and $\Pi_{T,N}$.   
   \end{algorithmic} 
\end{algorithm} 

The implementation efficiency of the weak POD-Greedy method relies on affine parametric structures. For example, the reduced order model \eqref{reducedmodel} for the parabolic problem \eqref{parabolic} allows a highly efficient solution process if the coefficient $a_\mu: \mathbb{I}\rightarrow L^\infty(\Omega)$ is of the form
\begin{align}\label{affinestructure}
    a_{\mu}(t_j)(\cdot)=\sum_{i=1}^{Q_a}w_i^j (\mu)a_i(\cdot),\quad 1\leq j\leq J,
\end{align}
for $a_i\in L^{\infty}(\Omega), 1\leq i\leq Q_a$.
In this case, parameter-independent terms could be pre-computed and stored in the offline stage, avoiding large and repeated calculations in the construction of the reduced basis subspace and evaluation of the POD-Greedy solutions. However, affine structures such as \eqref{affinestructure} are not available in many real-world applications.

The Empirical Interpolation Method (EIM) (see \cite{BarraultMadayNguyenPatera2004,MadayMulaTurinici2016}) is a greedy algorithm for approximating parametrized target functions by affinely separable functions.  When $a_\mu$ is not of the form in \eqref{affinestructure}, one could apply the EIM to each $a_{\mu}(t_j)(\cdot)\subset L^{\infty}(\Omega)$ and separate $\mu$ from the variable $x$. This approach is not efficient when the number $J$ of time layers is large. An alternative way is to construct the EIM approximation of $a_{\mu}(t)(x)$ on a one-dimension higher space-time domain $[0, T]\times\Omega$, which again is potentially expensive due to large dimensionality. 
In this section, we combine the POD in time with the EIM in space and introduce a EIM-POD-Greedy method for constructing affinely separable interpolants $\tilde{a}_\mu(t_j)(x)\approx a_\mu(t_j)(x)$ for $1\leq j\leq J$. The corresponding method is described in Algorithm \ref{alg:PODEIM}.

The construction of $\Pi_{n,m}$ in Algorithm \ref{alg:PODEIM} is motivated by the classical EIM and ensures the principal components $\left\{f_n^1,\ldots,f_n^m\right\}$ of the residual $r_n$ are well approximated by the EIM-POD-Greedy method interpolant. 
The EIM-POD-Greedy method in Algorithm \ref{alg:PODEIM} interpolates a function at $\{x_i\}_{i=1}^{mN}$, i.e., $(\Pi_Ng)(x_i)=g(x_i)$, $i=1, \ldots, mN$. Similarly to the classical EIM, the EIM-POD-Greedy method in Algorithm \ref{alg:PODEIM} satisfies the following properties: 
\begin{subequations}
\begin{align}
    &V_n={\rm span}\left\{ f_i^k\right\}_{1\leq i\leq n,1\leq k\leq m}\label{subspaceidentity},\\
    &\Pi_n^2=\Pi_n,\quad \Pi_{T,n}^2=\Pi_{T,n}\label{unisolvence}. 
\end{align}
\end{subequations}
    
To analyze the error of the EIM-POD-Greedy method, we introduce the set  $$\mathcal{F}:=\left\{a_{\mu}(t_j)\in V:\mu\in\mathcal{P},~0\leq j\leq J\right\}$$ and consider the norm $\|\Pi_n\|:=\sup_{0\neq v\in {\rm span}\{\mathcal{F}\}}\frac{\|\Pi_n v\|}{\|v\|}$ of the interpolation operator. 
The error of the EIM-POD-Greedy method in Algorithm \ref{alg:PODEIM} is measured by 
 \begin{equation*}
\hat{\sigma}_n:=\sup_{\mu\in\mathcal{P}}\|a_\mu-\Pi_{T,n-1}a_\mu\|_{V_T}.
 \end{equation*}
By using the argument in \cite{Li2024CGA}, we can derive convergence analysis of the EIM-POD-Greedy method based on the entropy numbers of the target function set.
\begin{theorem}\label{thm:PODEIMerror}
Let $\Lambda_n:=\|\Pi_n\|$ and $\theta_n:=\lambda_n^m/\lambda_n^1$ in Algorithm \ref{alg:PODEIM}. There exists a constant $C:=C(J,T,m)> 0$ such that for each $n\ge 1$, 
    \begin{align*}
        \hat{\sigma}_n\le C(1+\Lambda_{n-1})\Big(\prod_{i=1}^{n} (1+\Lambda_{i-1})\sqrt{\theta_i}^{-1}\Big)^{\frac{1}{n}}n^{\frac{1}{2}}\varepsilon_{mn}({\rm co}(\mathcal{\mathcal{F}})).
    \end{align*}
\end{theorem}
\begin{proof}
For any $a_{\mu}\in\mathcal{F}_T$, it holds that
    \begin{align}\label{eimprojection}
    \|a_{\mu}-P_{V_{T,n-1}}a_{\mu}\|_{V_T}=\inf_{w\in V_{T,n-1}}\|a_{\mu}-w\|_{V_T}\leq \|a_{\mu}-\Pi_{T,n-1}a_{\mu}\|_{V_T}.
\end{align}
Taking the supremum on both sides over $a_{\mu}\in\mathcal{F}_T$, we obtain
 \begin{align*}
     \sup_{a_{\mu}\in\mathcal{F}_T}\|a_{\mu}-P_{V_{T,n-1}}a_{\mu}\|_{V_T}\leq \sup_{a_{\mu}\in\mathcal{F}_T}\|a_{\mu}-\Pi_{T,n-1}a_{\mu}\|_{V_T}.
 \end{align*}
 By the definition of $\|\Pi_{n-1}\|$, it is straightforward to verify that $\|\Pi_{T,n-1}\|\leq\|\Pi_{n-1}\|$. Thus for $a_{\mu}\in\mathcal{F}_T$ we have 
\begin{align*}
    \|a_{\mu}-P_{V_{T,n-1}}a_{\mu}\|_{V_T}&\leq\|a_{\mu}-\Pi_{T,n-1}a_{\mu}\|_{V_T}\\
    &\leq(1+\Lambda_{n-1})\inf_{w\in V_{T,n-1}}\|a_{\mu}-w\|_{V_T}\notag\\
    &=(1+\Lambda_{n-1})\|a_{\mu}-P_{V_{T,n-1}}a_{\mu}\|_{V_T}.
\end{align*}
Therefore, the solution $a_{\mu_n}$ of the maximization problem
\begin{align*}
    \|a_{\mu_n}-\Pi_{T,n-1}a_{\mu_n}\|_{V_T}=\sup_{a_{\mu}\in\mathcal{F}_T}\|a_{\mu}-\Pi_{T,n-1}a_{\mu}\|_{V_T}
\end{align*}
in Algorithm \ref{alg:PODEIM} satisfies that 
\begin{align}\label{ProjectionAndInterpolation}
    \|a_{\mu_n}-P_{V_{T,n-1}}a_{\mu_n}\|_{V_T}\geq \frac{1}{1+\Lambda_{n-1}}\sup_{a_{\mu}\in\mathcal{F}_T}\|a_{\mu}-P_{V_{T,n-1}}a_{\mu}\|_{V_T}.
\end{align}
In addition, property \eqref{subspaceidentity} indicates that the interpolation bases form the same subspace as the POD bases. Therefore, the EIM-POD-Greedy method can be regarded as a POD-Greedy-type algorithm with $\gamma_i=\frac{1}{1+\Lambda_{i-1}}$. Using \eqref{eimprojection}, \eqref{ProjectionAndInterpolation} and Theorem \ref{thm:sigmaepsilon} with  $\gamma_i=\frac{1}{1+\Lambda_{i-1}}$, we conclude that for $1 \leq i \leq n$,
    \begin{align*}
        \|a_{\mu_n}-\Pi_{T,n-1}a_{\mu_n}\|_{V_T}&\le(1+\Lambda_{n-1})\|a_{\mu_n}-P_{V_{T,n-1}}a_{\mu_n}\|_{V_T}\\
        &\le C(1+\Lambda_{n-1})(\prod_{i=1}^n (1+\Lambda_{i-1})\sqrt{\theta_i}^{-1}\Big)^{\frac{1}{n}}n^{\frac{1}{2}}\varepsilon_{mn}({\rm co}(\mathcal{\mathcal{F}})).
    \end{align*}
    The proof is complete.
\end{proof}

When $\varepsilon_n({\rm co}(\mathcal{F}))$ is exponentially convergent to zero, e.g., $\mathcal{F}$ is analytic with respect to the parameter (see \cite{LiLi2024REIM}), convergence rate of the EIM-POD-Greedy method with multiple POD modes is higher than the single-mode one due to Theorem \ref{thm:PODEIMerror}.
%However, we acknowledge that this comparison needs further clarification, as the case $m>1$ inherently involves more basis vector in each iteration than the case $m=1$. To address this, we suggest that a more detailed comparison by considering both the number of POD modes used in each iteration and the role of exponential convergence. This would enable us to more clearly distinguish the effects of the exponential convergence of the error term from the influence of the additional modes included in the approximation.

\section{Numerical Experiments}\label{sec:NumExp}
In this section, we test the numerical performance of the weak POD-Greedy method as well as the EIM-POD-Greedy method. Throughout all numerical experiments, by $n$ and $N$ we denote the number of the POD-Greedy iterations and the dimension of the reduced basis subspace, respectively. In particular, $N=mn$ for the POD-Greedy method and the EIM-POD-Greedy method using $m$ POD modes.

\subsection{Weak POD-Greedy Method}\label{subsec:PODGreedyExp} First we consider the parabolic model \eqref{parabolic} with $\Omega=(-1,1)^2$, $T=1$, and 
$f = {\rm e}^{-t}\sin(\pi x)\sin(\pi y)$, $g = \sin(\pi x)\sin(\pi y)$.
The diffusion coefficient  $a_{\mu}=\mu\mathbbm{1}_{\Omega_1}+\mathbbm{1}_{\Omega_2}$ is time-independent, where $\mu$ takes its value in $\mathcal{P}$, a subset of 100 equidistantly-distributed points in $[1, 2]$, and  
$$\Omega_1=(-1,0]\times(-1,1),\quad\Omega_2=\Omega\backslash\Omega_1.$$
The high-fidelity fully discrete model is \eqref{fullydiscrete}, where $\tau=2^{-9}$ and $V_h$ is a linear finite element space based on a uniform triangular mesh with 263169 vertices.

Let $M_h$ be the finite element  matrix corresponding to the $H^1(\Omega)$ inner product, and $R_n$ the matrix with $R_n(:,j)$ representing the finite element function $u_{\tau,h,\mu_n}^j-P_{V_{n-1}}u_{\tau,h,\mu_n}^j$.
The POD step in Algorithm \ref{alg:PODGA} requires the $m$ leading eigen-pairs of $R_n^{\top}M_hR_n$, which is achieved by applying the function \texttt{svds} in MATLAB to $C_hR_n$, where $M_h=C_hC_h^\top$ is the Cholesky decomposition of $M_h$.

The reduced order model is given in \eqref{reducedmodel}. To efficiently implement the weak POD-Greedy method as explained in \eqref{estimator}, we make use of the following a posteriori error estimator 
\begin{align*}
\Delta_n(\mu):=\Delta_n(u_{\tau,h,\mu})=\sqrt{\tau\sum_{j=0}^J\|r_{j,n,\mu}\|_{V_h^*}^2},\end{align*}
where $r_{j,n,\mu}$ is a linear functional in $V_h^*$ defined by
\begin{align*}
    r_{j,n,\mu}(v_h):=(f(t_j),v_h)-\frac{(u_{\tau,n,\mu}^j-u_{\tau,n,\mu}^{j-1},v_h)}{\tau}-(a_\mu \nabla u_{\tau,n,\mu}^j,\nabla v_h).
\end{align*}
The dual norm of $r_{j,n,\mu}$ is 
\[
\|r_{j,n,\mu}\|_{V_h^*}:= \sup_{0\neq v_h\in V_h}\frac{r_{j,n,\mu}(v_h)}{\|v_h\|_{H^1(\Omega)}}.
\]
The error of the weak POD-Greedy method in Algorithm \ref{alg:PODGA} is
\begin{align*}
    E_N:=\sigma_{N/m}=\sup_{\mu\in\mathcal{P}}\|u_{h,\tau,\mu}-P_{V_{T,N/m-1}}u_{h,\tau,\mu}\|_{V_T}.
\end{align*}
The corresponding convergence history of $E_N$ with $1\leq m\leq4$ is presented in Figure \ref{fig:PODgreedytest} (left). It is observed that Algorithm \ref{alg:PODGA} with $m\geq 2$ is slightly less accurate than the one with $m=1$ using the reduced basis subspace of the same dimension. However, the computational cost of Algorithm \ref{alg:PODGA} with $m\geq 2$ is also much less than the single-mode one because it produces a $N$-dimensional reduced basis subspace with only $N/m$ POD-Greedy iterations, which significantly reduces the cost of the offline stage.

It is shown in Figure \ref{fig:PODgreedytest} (right) that a posteriori error estimator $\Delta_n$ and the numerical error $e_n$ are very close, which indicates that $\Delta_n$ is an efficient upper bound for the error $e_n$. 
Let $\lambda_n^1$, $\lambda_n^m$ denote the first and $m$-th singular values of $C_h R_n$, respectively. The eigenvalue ratio $\theta_n:=\lambda_n^m/\lambda_n^1$ for Algorithm \ref{alg:PODGA} with $m$ POD modes are listed in Table \ref{tab:PODGreedytheta}, showing fluctuations within a stable range. The non-monotonic behavior of $\theta_n$ with increasing $m$ is attributed to the fact that both the reduced basis subspace $V_n$ and the selected parameter $\mu_n$ vary with $m$, leading to variations in the eigenvalues. 

\begin{figure}[thp]
\centering
\subfloat{\includegraphics[width=0.5\textwidth]{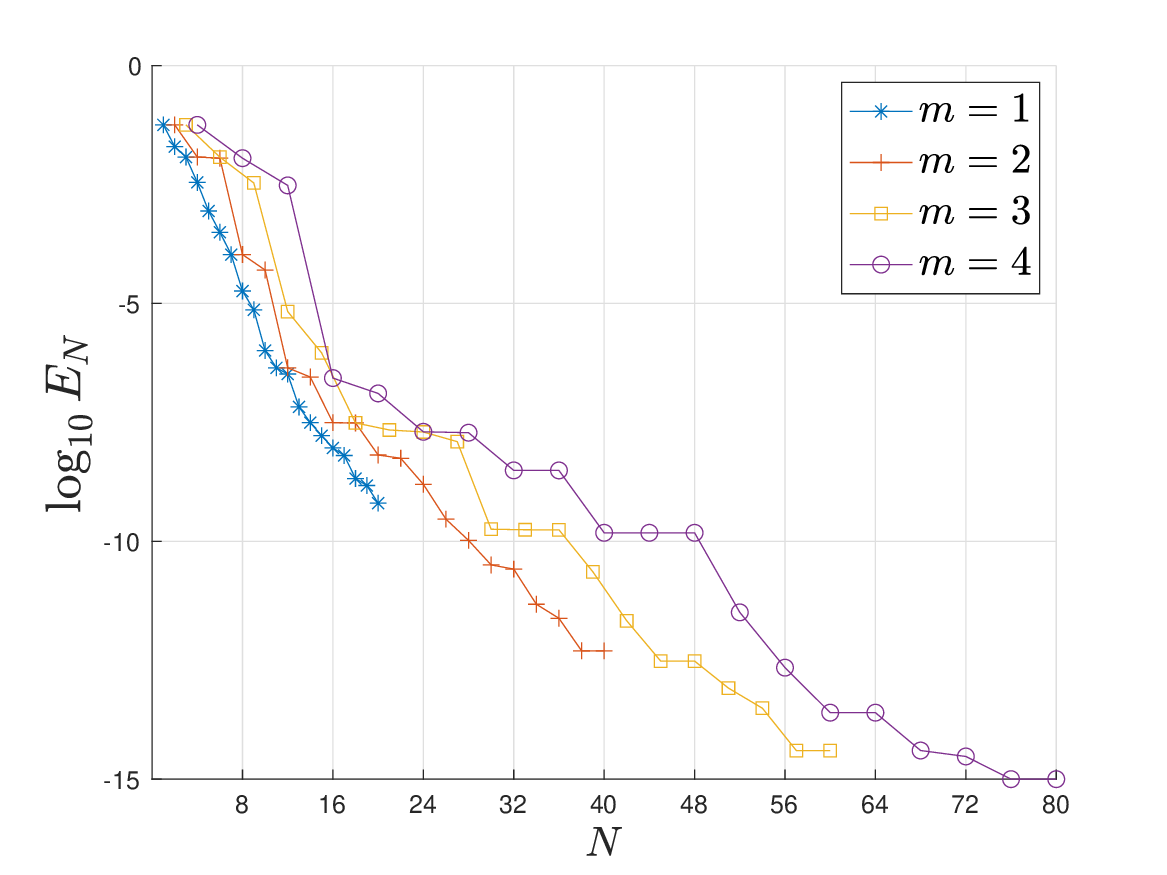}}
\hfill
\subfloat{\includegraphics[width=0.5\textwidth]{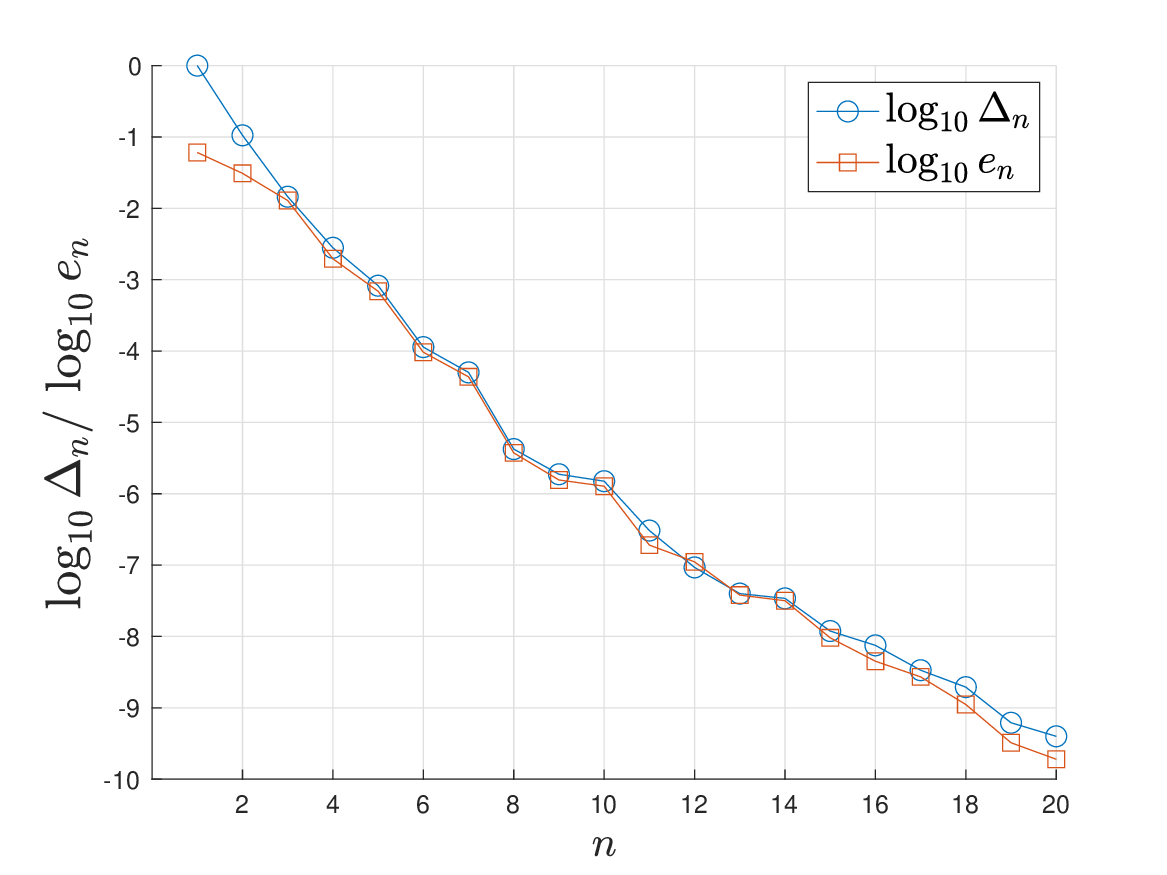}}
\caption{Convergence of $E_N$ with $m$ POD modes against the dimension $N$ of the reduced space (left); convergence history of $\Delta_n$ and $e_n=\|u_{\tau,h,\mu_n}-u_{\tau,n,\mu_n}\|$ with $m=1$ against the number of POD-Greedy iterations (right).}
\label{fig:PODgreedytest}
\end{figure}

\begin{table}[thp]
\centering
\begin{tabular}{|>{\centering\arraybackslash}m{2cm}||>{\centering\arraybackslash}m{2cm}|>{\centering\arraybackslash}m{2cm}|>{\centering\arraybackslash}m{2cm}|>{\centering\arraybackslash}m{2cm}|}
\hline
\hline
$n$  & $m=2$ & $m=3$ & $m=4$ \\
\hline
4  & 0.2415 & 0.0198 & 7.0044e-4\\
\hline
8  & 0.1312 & 0.0051 & 0.3881\\
\hline
12  & 0.4729 & 0.2482 & 0.0241\\
\hline 
16  & 0.3010 & 0.0307 & 0.5875\\
\hline
20  & 0.2759 & 0.0231 & 1.0895e-4\\
\hline
\end{tabular}
\caption{$\theta_n=\lambda_n^m/\lambda_n^1$ for the POD-Greedy method with $m$ POD modes at the $n$-th iteration.}
\label{tab:PODGreedytheta}
\end{table}

\subsection{EIM-POD-Greedy Method} In the second experiment, we consider the target function defined as 
\begin{align*}
a_{\bm{\mu}}(t)(x) = \frac{1}{\sqrt{(x-\mu_1)^2+(t-\mu_2)^2+1}},
\end{align*}
for $(t,x)\in I\times\Omega$ with $I=\Omega=(0,1)$, where $\bm{\mu}=(\mu_1,\mu_2)\in\hat{\mathcal{P}}\subset [0,1]^2$ contains 100 uniformly distributed parameters. The EIM-POD-Greedy method makes use of 128 time levels and the candidate set $\Sigma$ having  $100$ equidistributed points $\{\tilde{x}_i\}_{1\leq i\leq100}$ in $\Omega$. The classical EIM is implemented using 12800 points in the two-dimensional domain $I\times\Omega$.
The EIM-POD-Greedy method is implemented by using the $L^2(\mathbb{I};L^{\infty}(\Omega))$ norm as the $V_T$ norm. The POD step in Algorithm \ref{alg:PODEIM} is achieved by applying the MATLAB function \texttt{svds} to the residual matrix 
\begin{align}\label{PODEIMresidual}
    R_n=(a_{\bm{\mu}_n}(t_j)(\tilde{x}_i)-\Pi_{n-1}a_{\bm{\mu}_n}(t_j)(\tilde{x}_i))_{1\leq i\leq 100,1\leq j\leq 128}.
\end{align}

For the EIM-POD-Greedy method with $m=1$, we compute the a posteriori error estimator 
\begin{align*}
    \hat{\Delta}_n({\bm{\mu}}):=\sqrt{\tau\sum_{j=0}^J|a_{\bm{\mu}}(t_j)(x_n)-\Pi_{n-1}a_{\bm{\mu}}(t_j)(x_n)|^2}.
\end{align*}
The effectivity of $\hat{\Delta}_n({\bm{\mu}})$ is measured by $$\eta_n({\bm{\mu}}):=\frac{\hat{\Delta}_n({\bm{\mu}})}{\|a_{\bm{\mu}}-\Pi_{T,n-1}a_{\bm{\mu}}\|_{V_T}}$$ 
and the average effectivity index is  $\overline{\eta}_n:=\frac{1}{|\mathcal{P}|}\sum_{{\bm{\mu}}\in\mathcal{P}}\eta_N({\bm{\mu}})$, which is expected to be close to unity. Let $\kappa_n:=\kappa(B_{n,1})$ denote the $\ell^2$ condition number of $B_{n,1}$. Let $\left\{w_i\right\}_{1\leq i\leq n}\subset V_n$ be the interpolation basis of $V_n$ given by 
\begin{align*}
    w_i(x_j)&=\delta_{ij},\quad\quad 1\leq j\leq n.
\end{align*}
It is straightforward to see that  
\begin{align*}
    \|\Pi_n\|=\Lambda_n\leq \tilde{\Lambda}_n:=\sup_{x\in\Omega}\sum_{i=1}^n|w_i(x)|.
\end{align*}

Table \ref{tab:PODEIMconstant} records the values of $ \overline{\eta}_n, \kappa_n, \tilde{\Lambda}_n$. The average effectivity $\overline{\eta}_n$, being close to unity as anticipated, confirms that $\hat{\Delta}_n(\bm{\mu})$ is an efficient error estimator. The modest growth of $\tilde{\Lambda}_n$ provides a mild upper bound for the norm of $\Pi_n$. The eigenvalue ratios $
\theta_n:=\lambda_n^m/\lambda_n^1$ for the EIM-POD-Greedy method with $m$ POD modes are presented in Table \ref{tab:PODEIMtheta}, where $\lambda_n^1$ and $\lambda_n^m$ denote the first and $m$-th singular values of the residual matrix $R_n$ in \eqref{PODEIMresidual}. The non-monotonic behavior of $\theta_n$ as $m$ increases is due to the same reason as that explained for $\theta_n$ in the POD-Greedy method. 

The error of the EIM-POD-Greedy method is 
\begin{align*}
    \widehat{E}_N:=\hat{\sigma}_{N/m}=\sup_{\mu\in\hat{\mathcal{P}}}\|a_{\mu}-\Pi_{T,N/m-1}a_{\mu}\|_{V_T}.
\end{align*}
Convergence history of the EIM-POD-Greedy error $\hat{E}_N$ is presented in Figure \ref{fig:PODEIMtest} (left). Similarly to the numerical results in Section \ref{subsec:PODGreedyExp}, the EIM-POD-Greedy method with multiple POD modes exhibits a higher efficiency than the single-mode one. In Figure \ref{fig:PODEIMtest} (right), we show the convergence history of the EIM-POD-Greedy method and the classical EIM, highlighting the superior efficiency of the EIM-POD-Greedy method. 
\begin{table}[ht]
    \centering
    \begin{tabular}{|>{\centering\arraybackslash}m{2cm}||>{\centering\arraybackslash}m{2cm}|>{\centering\arraybackslash}m{2cm}|>{\centering\arraybackslash}m{2cm}|>{\centering\arraybackslash}m{2cm}|}
    \hline
    \hline
    $n$  & $\overline{\eta}_n$ & $\kappa_n$ & $\tilde{\Lambda}_n$ \rule{0pt}{12pt}\\
    \hline
    4  & 0.9911 & 3.2549 & 1.5990\\
    \hline
    8  & 0.9974 & 5.9552 & 2.8051\\
    \hline
    12  & 0.9985 & 9.1616 & 5.3882\\
    \hline 
    16  & 0.9998 & 15.3889 & 4.8197\\
    \hline
    20  & 0.9844 & 19.2868 & 6.5005\\
    \hline
    \end{tabular}
    \caption{$\overline{\eta}_n$, $\kappa_n$, $\tilde{\Lambda}_n$ for the EIM-POD-Greedy method with $m=1$ at the $n$-th EIM-POD-Greedy iteration.}
    \label{tab:PODEIMconstant}
\end{table}
\begin{figure}[ht]
    \centering
\subfloat{\includegraphics[width=0.5\textwidth]{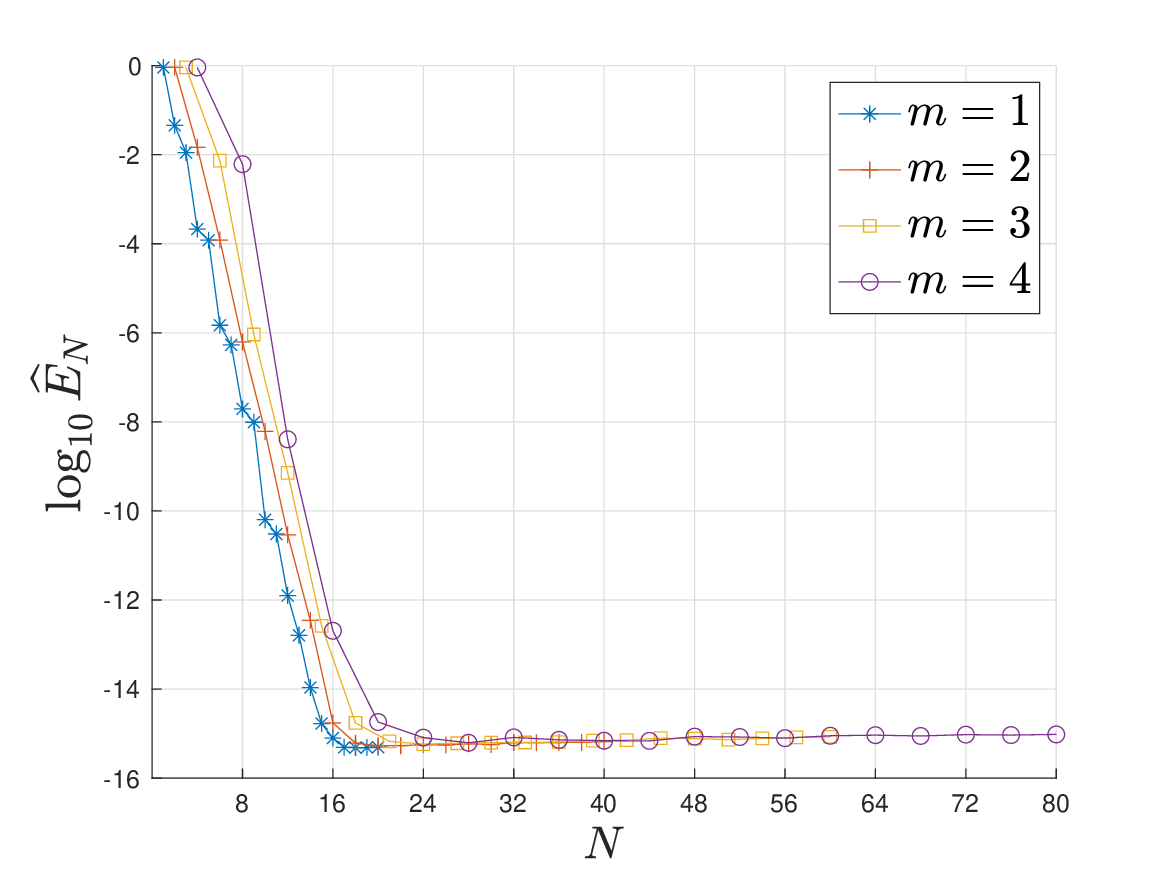}}
    \hfill
    \subfloat{\includegraphics[width=0.5\textwidth]{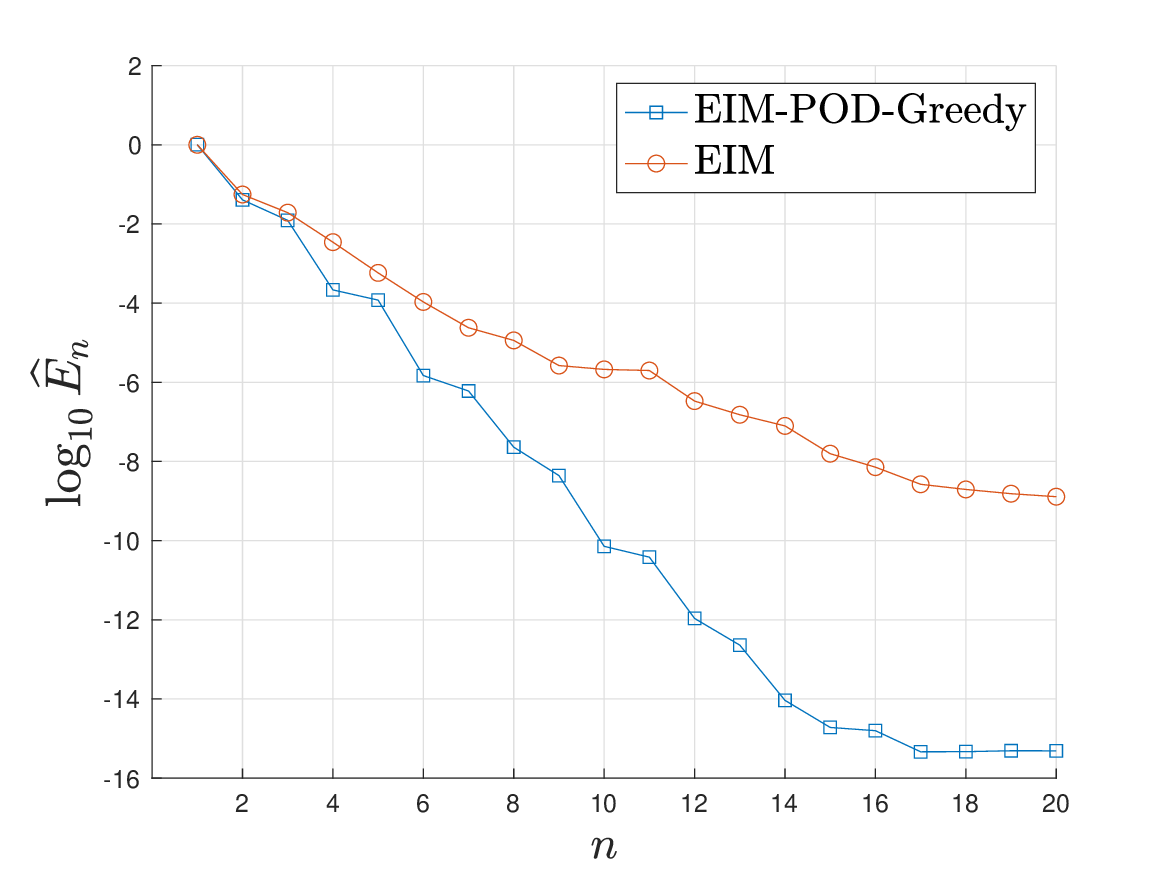}}
    \caption{Convergence of the EIM-POD-Greedy method with $m=1,2,3,4$ POD modes against the dimension of the reduced subspace (left); convergence of the EIM and the EIM-POD-Greedy method with $m=1$ against the number of EIM-POD-Greedy iterations (right).}
    \label{fig:PODEIMtest}
\end{figure}
    \begin{table}[ht]
    \centering
    \begin{tabular}{|>{\centering\arraybackslash}m{2cm}||>{\centering\arraybackslash}m{2cm}|>{\centering\arraybackslash}m{2cm}|>{\centering\arraybackslash}m{2cm}|>{\centering\arraybackslash}m{2cm}|}
    \hline
    \hline
    $n$  & $m=2$ & $m=3$ & $m=4$ \\
    \hline
    4  & 0.0495 & 3.3206e-4 & 0.0021\\
    \hline
    8  & 0.2289 & 0.4512 & 0.5746\\
    \hline
    12  & 0.6703 & 0.5661 & 0.5031\\
    \hline 
    16  & 0.6484 & 0.6071 & 0.6661\\
    \hline
    20  & 0.9141 & 0.3978 & 0.5305\\
    \hline
    \end{tabular}
    \caption{$\theta_n=\lambda_n^m/\lambda_n^1$ for the EIM-POD-Greedy method with $m$ POD modes at the $n$-th iteration.}
    \label{tab:PODEIMtheta}
\end{table}

\section*{Acknowledgements} This work was partially supported by the National Science Foundation of China (no.~12471346) and the Fundamental Research Funds for the Zhejiang Provincial Universities (no. 226-2023-00039).

\bibliographystyle{amsplain}

\providecommand{\bysame}{\leavevmode\hbox to3em{\hrulefill}\thinspace}
\providecommand{\MR}{\relax\ifhmode\unskip\space\fi MR }
\providecommand{\MRhref}[2]{%
  \href{http://www.ams.org/mathscinet-getitem?mr=#1}{#2}
}
\providecommand{\MRhref}[2]{#2}

\end{document}